\documentclass[12pt]{amsart}
\usepackage{graphicx}
\usepackage{amsmath}
\usepackage{amssymb}
\usepackage{setspace}
\newtheorem{thm}{Theorem}[section]

\newtheorem{mrem}{Remark}
\newtheorem{mthm}{Theorem}
\newtheorem{mconjecture}{Conjecture}
\newtheorem{lem}[thm]{Lemma}

\theoremstyle{definition}

\theoremstyle{remark}

\numberwithin{equation}{section}

\begin{document}
\title[The fractional p-Laplacian]{Maximum principles and monotonicity of solutions for fractional p-equations in unbounded domains}

\author{Zhao Liu$^\dag$, Wenxiong Chen $^*$$^\ddag$}

\address{$^\dag$ School of Mathematics and Computer Science\\Jiangxi Science and Technology Normal University\\Nanchang 330038, P. R. China
and Department of Mathematical Sciences\\Yeshiva University\\New York, NY, USA}
\email{liuzhao@mail.bnu.edu.cn}

\address{$^\ddag$Department of Mathematical Sciences\\Yeshiva University\\New York, NY,10033 USA}
\email{wchen@yu.edu}

\thanks{$^*$ Corresponding author: Wenxiong Chen at wchen@yu.edu. \\
Zhao Liu is supported by the NNSF of China (No. 11801237).}

\begin{abstract}
In this paper, we consider the following non-linear equations in unbounded domains $\Omega$ with exterior Dirichlet condition:
\begin{equation*}\begin{cases}
(-\Delta)_p^s u(x)=f(u(x)), & x\in\Omega,\\
u(x)>0, &x\in\Omega,\\
u(x)\leq0, &x\in \mathbb{R}^n\setminus \Omega,
\end{cases}\end{equation*}
where $(-\Delta)_p^s$
is the fractional p-Laplacian defined as
\begin{equation}
(-\Delta)_p^s u(x)=C_{n,s,p}P.V.\int_{\mathbb{R}^n}\frac{|u(x)-u(y)|^{p-2}[u(x)-u(y)]}{|x-y|^{n+s p}}dy
\label{0}
\end{equation}
with $0<s<1$ and $p\geq 2$.

We first establish a {\em maximum principle in unbounded domains involving the fractional p-Laplacian} by estimating the singular integral in (\ref{0}) along a sequence of approximate maximum points.  Then, we
obtain the asymptotic behavior of solutions far away from the boundary. Finally, we develop a sliding method for the fractional p-Laplacians and apply it to derive the monotonicity and uniqueness of solutions.

There have been similar results for the regular Laplacian \cite{BCN1} and for the fractional Laplacian \cite{DSV}, which are linear operators. Unfortunately, many approaches there no longer work for the fully non-linear fractional p-Laplacian here. To circumvent these difficulties, we introduce several new ideas,
which enable us not only to deal with non-linear non-local equations, but also to remarkably weaken the conditions on $f(\cdot)$ and on the domain $\Omega$.

We believe that the new methods developed in our paper can be widely applied to many problems in unbounded domains involving non-linear non-local operators.
\end{abstract}
\maketitle {\small {\bf Keywords:} Fractional p-Laplacians, maximum principles, unbounded domains, sliding methods, asymptotic behavior, monotonicity of solutions, uniqueness.\\

{\bf 2010 MSC} Primary: 35R11, 35J91; Secondary: 35B06, 35B65.}

\section{Introduction}

In this paper, we investigate qualitative properties of solutions for the nonlinear problem
\begin{equation}\label{fractional000}\begin{cases}
(-\Delta)_p^s u(x)=f(u(x)), &x\in\Omega,\\
u(x)>0, &x\in\Omega,\\
u(x)\leq0, &x\in \mathbb{R}^n\setminus \Omega,
\end{cases}\end{equation}
where $\Omega$ is the region above the graph of a continuous function $\varphi:\mathbb{R}^{n-1}\rightarrow\mathbb{R}$, i.e.
\begin{equation*}
\Omega:=\{x=(x',x_n)\in\mathbb{R}^n\mid x_n>\varphi(x')\} \ \ \text{with}\ \ x'=(x_1,x_2,\cdots,x_{n-1})\in \mathbb{R}^{n-1}.
\end{equation*}

$(-\Delta)_p^s$  is the fractional p-Laplacian defined as
\begin{equation*}\begin{split}
(-\Delta)_p^s u(x)&=C_{n,s,p}P.V.\int_{\mathbb{R}^n}\frac{|u(x)-u(y)|^{p-2}[u(x)-u(y)]}{|x-y|^{n+s p}}dy\\
&=C_{n,s,p}\lim_{\varepsilon\rightarrow0}\int_{\mathbb{R}^n\setminus {B_\varepsilon(x)}}\frac{|u(x)-u(y)|^{p-2}[u(x)-u(y)]}{|x-y|^{n+s p}}dy,
\end{split}\end{equation*}
where P.V. stands for the Cauchy principal value, and throughout this paper, we assume that $s\in(0,1)$ and $p\geq 2$.

In order the integral to make sense, we require that
$$u\in C_{loc}^{1,1}\cap \mathcal{L}_{sp}$$
with
$$\mathcal{L}_{sp}:=\{u\in L_{loc}^{p-1}\mid \int_{\mathbb{R}^n}\frac{|u(x)|^{p-1}}{1+|x|^{n+sp}}dx<\infty\}.$$

In the special case when $p=2$, $(-\Delta)^s_p$
becomes the well-known fractional Laplacian
 $(-\Delta)^s$. And one can show that, as $s \rightarrow 1$, the fractional p-Laplacian converges to the regular p-Laplacian:
 $$ (-\Delta)^s_p u(x) \rightarrow -\Delta_p u(x) := - div(|\bigtriangledown u(x)|^{p-2} \bigtriangledown u(x)).$$

 The non-local nature of these operators make them difficult to study. To circumvent this,
 Caffarelli and Silvestre \cite{CS} introduced the extension method which turns the
non-local problem involving the fractional Laplacian into a local one in higher dimensions.
This method has been applied successfully to study equations involving the fractional
Laplacian, and a series of fruitful results have been obtained (see \cite{BCDS,CZ} and the references therein).
One can also use the integral equations method, such as the method of
moving planes in integral forms and regularity lifting to investigate equations involving
the fractional Laplacian by first showing that they are equivalent to the corresponding
integral equations (see \cite{CFY,CLO,CLO1}).

\medskip
However,
so far as we know, besides the fractional Laplacian, there has not been any extension
methods that works for other non-local operators, such as the uniformly elliptic
non-local operators and fully non-linear non-local operators
(see \cite{CS1} for the introductions of these operators) including
the fractional p-Laplacian.
In \cite{CLL}, Chen, Li and Li introduced the direct method of moving planes for the fractional Laplacian which has been applied to obtain symmetry, monotonicity, and
non-existence of solutions for various semi-linear equations involving the fractional Laplacian.
In \cite{CLLG}, Chen, Li and Li refined this direct approach, so
that it can be applied to fully nonlinear nonlocal problem in the case the operator is
non-degenerate in certain sense. In order to investigate the degenerate fractional
p-Laplacian, Chen and Li \cite{CL2} introduced some new ideas, among which
a significant one is a variant of the Hopf Lemma, the key boundary estimate,
which plays the role of the narrow region principle in the second step of the method of
moving planes.
For more applications about this direct method for various non-local problems,
please see \cite{CLL1,CLZ} and the references therein.

\medskip

It is well-known that \emph{maximum principles} play  fundamental roles in the study of elliptic
partial differential equations, it is also a powerful tool in carrying out
the \emph{method of moving planes} to derive symmetry, monotonicity, and non-existence
of solutions. Recently, due to their broad applications to various branches of
sciences, a lot of attention has been turned to the
non-linear equations involving fractional Laplacians and other non-local operators, including the fully non-linear non-local fractional p-Lapcians. In order to further investigate these non-local equations, here we establish a fractional p-Laplacian version of the \emph{maximum principle}
in unbounded domains without assuming any asymptotic behavior of the solutions near infinity, which can be applied to establish
qualitative properties, such as symmetry and monotonicity for solutions of
fractional p-Laplacian equations.

\medskip
Our first result is the following.
\begin{mthm}\label{unbonudedthm}
Let $D$ be a open set in $\mathbb{R}^n$, possibly unbounded and disconnected. Assume that $\overline{D}$ is disjoint
from the closure of an infinite open domain $\Sigma$ satisfying
\begin{equation}\label{jixian11}
\underline{\lim}_{j\rightarrow\infty}
\frac{\big|\big(B_{2^{j+1}r}(x)\setminus{B_{2^{j}r}(x)}\big)\cap\Sigma\big|}{|B_{2^{j+1}r}(x)\setminus B_{2^{j}r}(x)|}\geq c_0,\ \ \forall x\in\mathbb{R}^n,
\end{equation}
for some $c_0>0$ and $r>0$. Suppose  $u(x)$ is in $C_{loc}^{1,1}\cap \mathcal{L}_{sp}$, bounded from above,
and satisfies
\begin{equation}\begin{cases}
(-\Delta)_p^s u(x)+c(x)u(x)\leq0,&x\in D,\\
u(x)\leq0, & x\in \mathbb{R}^n\setminus D,
\end{cases}
\label{a2}
\end{equation}
for some nonnegative function $c(x)$.

Then $u(x)\leq0$ in $D$.
\end{mthm}

\medskip
 A similar \emph{maximum principle} in unbounded domains in the classical case involving the regular  Laplacian (when $s=1$, $p=2$) was obtained by Berestycki, Caffarelli and Nirenberg \cite{BCN1}.
Birindelli and Prajapat \cite{BP} extended the {\em maximum principle} to the Heisenberg group.
For the fractional Laplacian (when $s\in(0,1)$, $p=2$), Dipierro, Soave and Valdinoci \cite{DSV} proved the same results based on \emph{growth lemmas} established by De Giorgi \cite{De} and Silvestre \cite{S}
respectively.
\medskip

In both of the above articles \cite{BCN1, DSV}, the authors assumed that the complement of $\overline{D}$ contains  an infinite open cone $\Sigma$. One may call this an {\em exterior cone condition}. It is
easy to check that the infinite open cone satisfies \eqref{jixian11} in our Theorem \ref{unbonudedthm}. Actually, one can see that our condition  \eqref{jixian11} is much weaker than the {\em exterior cone condition}. There are many domains $\overline{D}$ whose complement do not contain an infinite cone. To illustrate this, we list the following two simple examples of such domains.
\medskip

(a) $D=\{x\in\mathbb{R}^n \mid 2i<x_n<2i+1,\ i=0,\pm1,\pm2,\cdots \}$.
\medskip

(b) $D =\{x\in\mathbb{R}^n \mid 2i<|x|<2i+1,\ i=0,1,2,\cdots \}$.
\medskip

Obviously, none of the above two domains $D$ satisfy the {\em exterior cone condition}. Since our Theorem \ref{unbonudedthm} includes the case when $p=2$, it improves the result in \cite{DSV} by weakening the condition on the domains.
\medskip

We would like to mention that the operators $-\Delta$ and $(-\Delta)^s$ considered in \cite{BCN1} and \cite{DSV} respectively are linear ones, while the fractional p-Laplacian $(-\Delta)^s_p$ in this paper is fully non-linear. Hence the methods in \cite{BCN1} and \cite{DSV} can no longer be applied here. To deal with such non-local non-linear operators, we introduce new ideas.
\medskip

Usually, to prove a maximum principle on bounded domains, or on unbounded domains assuming that the solutions vanishes near infinity, one derives contradictions at a maximum point. However, on unbounded domains without imposing any asymptotic conditions on the solution $u$, the maximum value of $u$ may not be attained, and a maximizing sequence may tend to infinity.
To circumvent this difficulty, we estimate the singular integral defining $(-\Delta)^s_p u$ along a sequence of approximate maximum points to derive a contradiction if, in Theorem \ref{unbonudedthm}, $\sup_D u(x) > 0$. It turns out that this approach is quite simple, and it also applies to the case of fractional Laplacian (when $p=2$).
We believe that this method will become a very useful tool to investigate  many other non-linear equations involving general non-linear non-local operators.
\medskip

The \emph{moving plane method} and the \emph{sliding method} are techniques that have been used
in recent years to establish qualitative properties of solutions of
non-linear elliptic equations such as symmetry, monotonicity, and non-existence. In fact,
the \emph{method of moving planes} was initially invented by Alexanderoff in the early 1950s. Later, it was further developed by Serrin \cite{Se}, Gidas, Ni and Nirenberg \cite{GNN,GNN1}, Caffarelli, Gidas and Spruck \cite{CGS}, Chen and Li \cite{CL}, Li and Zhu \cite{LZ}, Chang and Yang \cite{CY}, Lin \cite{Lin} and many others. For more literatures about the method of moving planes,
please refer to \cite{CFY,CL1,CLZ,DFHQW,DFQ,DLL,DQ,FC,LD,LZ1,LZ2,LZ3,MZ,WX,ZCCY} and the references therein.
The \emph{sliding method} was introduced by Berestycki, Caffarelli and Nirenberg \cite{BCN1,BCN2,BN}, which is slightly
different from the\emph{ method of moving plane}, it is used to compare the solution with its translation
rather than its reflection.
The \emph{sliding method} was also successful in obtaining symmetry and monotonicity of solutions for many kind of domains (see \cite{BCN2,BHM}).
\medskip

We consider the following fractional p-Laplacian equation
\begin{equation}\label{fractionalbianjiewai}\begin{cases}
(-\Delta)_p^s u(x)=f(u(x)),&x\in\Omega,\\
u(x)>0, &x\in\Omega.\\
\end{cases}\end{equation}

\medskip

As preparations to carry out the {\em sliding method} along $x_n$-direction, we obtain the following two theorems, which may also be applied to other situations.
\begin{mthm}\label{xiaoyu1}
Let $u\in C_{loc}^{1,1}\cap \mathcal{L}_{sp}$ be a bounded solution of \eqref{fractionalbianjiewai}.
Assume that $f$ is continuous and satisfies

(a)  There exists $\mu>0$ such that $f(t)>0$
on $(0,\mu)$, and $f(t)\leq0$ for $t\geq\mu$.

Suppose that $u(x)<\mu$, $x\in \mathbb{R}^n\setminus \Omega$, then $$u(x) <\mu  \mbox{ for all } x \in \Omega.$$
\end{mthm}

\begin{mthm}\label{mainthm11}
Let $u\in C_{loc}^{1,1}\cap \mathcal{L}_{sp}$ be a bounded solution of \eqref{fractionalbianjiewai} with
$$0\leq u(x)<\mu, \;\; x\in \mathbb{R}^n\setminus \Omega.$$
Assume that $f$ is continuous, satisfies condition (a)
and for some $0<t_0<t_1<\mu$,

(b) $f(t)\geq\delta_0t$ on $[0,t_0]$ for some $\delta_0>0$, and

(c) $f(t)$ is nonincreasing on $(t_1,\mu)$.

Then $u(x)\rightarrow\mu$ uniformly in $\Omega$ as $dist(x,\partial\Omega)\rightarrow\infty$.
\end{mthm}

Let us point out that Theorem \ref{xiaoyu1} and Theorem \ref{mainthm11} are closely related to  the well-known {\em De Giorgi conjecture}:

\begin{mconjecture}
(De Giorgi \cite{De2}). If $u$ is a solution of
$$- \Delta u = u-u^3,$$  such that
 \begin{equation}  |u|\leq1  \mbox{ in } \mathbb{R}^n, \lim_{x_n\rightarrow\pm\infty} u(x',x_n)=\pm 1
\mbox{ for all } x'\in \mathbb{R}^{n-1}, \label{a3}
\end{equation} and
\begin{equation}
\frac{\partial u}{\partial x_n}> 0.
\label{a4}
\end{equation}

Then
there exists a vector $a\in \mathbb{R}^{n-1}$ and a function $u_1$: $\mathbb{R}\rightarrow\mathbb{R}$ such that
$$u(x',x_n)=u_1(a\cdot x'+x_n) \mbox{ in } \mathbb{R}^n.$$
\end{mconjecture}

If we replace $-\Delta$ by $(-\Delta)^s_p$, take $\mu=1$, and $f(u)=u-u^3$ as in the {\em De Giorgi's conjecture},
then conditions (a)-(c) in Theorem \ref{mainthm11} are satisfied. Hence we
derive that
$$u<1 \mbox{ in } \Omega \mbox{ and }  u(x)\rightarrow1 \mbox{ uniformly in } \Omega \mbox{ as } dist(x,\partial\Omega)\rightarrow\infty.$$
Therefore, we can replace condition (\ref{a3}) by
$$u>0 \mbox{ for } x_n>M \;\;  (\mbox{take } \varphi(x')\equiv M).$$

Based on the above two theorems, we will apply the {\em sliding method} to obtain the monotonicity of solutions for the following problem.

\begin{equation}\label{fractional}\begin{cases}
(-\Delta)_p^s u(x)=f(u(x)),&x\in\Omega,\\
u(x)>0,&x\in\Omega,\\
u(x)=0, &x\in \mathbb{R}^n\setminus \Omega,
\end{cases}
\end{equation}
where $\Omega$ satisfies the uniform two-sided ball condition (the exterior and interior ball conditions).
\medskip

We prove

\begin{mthm}\label{mainthm}
Let $u\in C_{loc}^{1,1}\cap \mathcal{L}_{sp}$ be a bounded solution of \eqref{fractional}.
Assume that $f$ is a continuous function and satisfies conditions (a)-(c) for some $0<t_0<t_1<\mu$.

Then $u$ is strictly monotone increasing in $x_n$.

Furthermore, the bounded solution of (\ref{fractional}) is unique.
\end{mthm}
\medskip

\begin{mrem}\label{shiyongp=2}
Our results in this paper adapt to the case of the fractional Laplacian where $p=2$.
As prototype in Theorem \ref{mainthm}, we may take $f(u)=u-u^3$ or $f(u)=u-u^2$. Then equation \eqref{fractional}
is the well-known fractional Allen-Cahn equation or the fractional Fisher-Kolmogorov equation,
which have been widely studied by many authors (please see \cite{CW,DGYW,SV} and the references therein).
\end{mrem}
\medskip
\begin{mrem}\label{remarkaa}
Theorem \ref{mainthm} was proved by Berestycki, Caffarelli and Nirenberg \cite{BCN1} for $s=1$ and $p=2$,
and Dipierro, Soave and Valdinoci \cite{DSV} for $s\in(0,1)$ and $p=2$ respectively.
They all assumed $f(\cdot)$ to be globally Lipschitz continuous. In this paper we only require $f(\cdot)$ to be continuous, which is weaker than the condition in
the classical results established by
Berestycki, Caffarelli and Nirenberg \cite{BCN1} and Dipierro, Soave and Valdinoci \cite{DSV}.
This is mainly because we employ a new and different idea here.
\end{mrem}

To illustrate the major differences between the traditional approach and our approach, let
$$ x =(x', x_n), \; u_\tau (x) = u(x',  x_n+\tau), \; \mbox{ and } w_\tau (x) = u(x) - u_\tau (x).$$
To obtain the result in Theorem \ref{mainthm}, one first needs to show that
$$w_\tau (x) \leq 0, \;\;  \forall \, \tau > 0, x \in \Omega.$$
This is achieved via a contradiction argument.
Suppose $\sup w_\tau = A > 0$, then there exists a sequence $\{x^k\}$, such that
$$w_\tau (x^k) \rightarrow A, \; \mbox{ as } k \rightarrow \infty.$$
Making the translation $w_\tau^k (x) = w_\tau (x+x^k)$,
in the linear operator case as in \cite{BCN1} and \cite{DSV},  they obtained
$$ (-\Delta)^s w_\tau^k(x) = c_k(x) w_\tau^k (x). $$
Here $c_k(x)$ are uniformly bounded due to the global Lipshitz continuity assumption on $f$. Based on this, they were able to show that
$$ w_\tau^k(x) \rightarrow w_\tau^\infty (x) \mbox{ and } (-\Delta)^s w_\tau^k (x) \rightarrow (-\Delta)^s w_\tau^\infty (x),$$
and therefore
$$ (-\Delta)^s w_\tau^\infty (x) = c_\infty(x) w_\tau^\infty (x) \mbox{ with } w_\tau^\infty (0) = A >0. $$
Then, they were able to derive a contradiction based on the properties of the solutions of the above equation. 

In our nonlinear operator case, the first difficulty is
$$(-\Delta)^s_p u(x) - (-\Delta)^s_p u_\tau (x) \neq (-\Delta)^s_p w_\tau (x).$$
Hence the simple \emph{maximum principle} such as Theorem \ref{unbonudedthm} can not be applied directly.
We will modify it in the proof of Theorem \ref{mainthm}.

The second difficulty is more subtle. So far, there have been very few regularity results on the solutions for fractional p-equations, the best we know is that the solutions $u$ are uniformly H\"{o}lder continuous if both $u$ and $f(u)$ are bounded. These are far from sufficient to guarantee the convergence of
$$ (-\Delta)^s_p u^k (x) - (-\Delta)^s_p u_\tau^k (x),$$
which requires $\{u^k\}$ to be uniformly $C^{1,1}$.

To circumvent this difficulty, instead of estimating along a sequence of equations in the whole domain $\Omega$, we estimate the singular integrals defining $(-\Delta)^s_p u (x) - (-\Delta)^s_p u_\tau (x)$ only along a sequence of points, the approximate maximum points $x^k$. This new idea not only enable us to deal with the situation where the lack of the regularity result is known, but also enable us to weaken the condition on the nonlinearity $f(u)$.

Finally, we consider a special case where $\Omega$ is an upper half space:
\begin{equation}\label{shangbankongjian}\begin{cases}
(-\Delta)_p^s u(x)=f(u(x)), &x\in\mathbb{R}^n_+,\\
u(x)>0, &x\in\mathbb{R}^n_+,\\
u(x)=0, &x\in \mathbb{R}^n\setminus {\mathbb{R}^n_+}.
\end{cases}\end{equation}
\medskip

For this particular domain, we are able to use the \emph{sliding method} in any direction to obtain a stronger result.

\begin{mthm}\label{bankongjian}
Suppose that $u\in C_{loc}^{1,1}\cap \mathcal{L}_{sp}$ be a bounded solution of \eqref{shangbankongjian}.
Assume that  $f$ is continuous and satisfies conditions (a)-(c) for some $0<t_0<t_1<\mu$.

Then $u$ is strictly monotone increasing in $x_n$, and moreover it depends on $x_n$ only.

Furthermore the bounded solution of \eqref{shangbankongjian} is unique.
\end{mthm}

\medskip

The rest of our paper is organized as follows. In section 2, we prove the \emph{maximum principles} in unbounded domains
and hence establish Theorem \ref{unbonudedthm}. Based on the \emph{maximum principles}, we obtain Theorem \ref{xiaoyu1}.
In section 3, we carry out our proof of Theorem \ref{mainthm11} by using a \emph{sliding method} on ball regions.
In section 4, we prove the monotonicity and uniqueness by estimating the singular integrals along the
\emph{approximate maximum points} in the process of \emph{sliding} and
thus obtain Theorem \ref{mainthm}.
Section 5 is devoted to proving Theorem \ref{bankongjian}.

\medskip
In the following, we will use $C$ to denote a general positive constant that may depend on $n$, $s$ and $p$, and whose value may differ from line to line.

\section{The proof of Theorem \ref{unbonudedthm} and Theorem \ref{xiaoyu1}}
In this section, we establish the following \emph{maximum principles} in unbounded domains.

\begin{thm}\label{unbonudedthma}
Let $D$ be a open set in $\mathbb{R}^n$, possibly unbounded and disconnected. Assume that $\overline{D}$ is disjoint
from the closure of an infinite open domain $\Sigma$ satisfying
\begin{equation}\label{jixian}
\underline{\lim}_{j\rightarrow\infty}\frac{\big|\big(B_{2^{j+1}r}(x)\setminus{B_{2^{j}r}(x)}\big)\cap\Sigma\big|}
{|B_{2^{j+1}r}(x)\setminus B_{2^{j}r}(x)|}\geq c_0,\ \ \forall x\in\mathbb{R}^n,
\end{equation}
for some $c_0>0$ and $r>0$. Suppose  $u(x)$ is in $C_{loc}^{1,1}\cap \mathcal{L}_{sp}$, bounded from above,
and satisfies
\begin{equation}\label{unbdd}\begin{cases}
(-\Delta)_p^s u(x)+c(x)u(x)\leq0,& x\in D,\\
u(x)\leq0, &x\in \mathbb{R}^n\setminus D,
\end{cases}
\end{equation}
for some nonnegative function $c(x)$.

Then $u(x)\leq0$ in $D$.
\end{thm}

\medskip

\begin{proof}
Suppose on the contrary, there is some points $x$ such that $u(x)>0$ in $D$, then
\begin{equation}\label{larg}
0<A:=\sup_{x\in\mathbb{R}^n}u(x)<\infty.
\end{equation}
There exists sequences $x^k\in D$ and $\gamma_k\rightarrow1$ $(\gamma_k\in (0,1))$ as $k\rightarrow\infty$ such that
\begin{equation}\label{key}
u(x^k)\geq \gamma_k A.
\end{equation}
Let
\begin{equation}
\Phi(x)=\begin{cases}
&c_ne^{\frac{1}{|x|^2-4}}, \ \ |x|< 2,\\
&0,\ \ \ \ \ \ \ \ \ \ \ |x|\geq 2.
\end{cases}\end{equation}
It is easy to check that $\Phi$ is radially decreasing from the origin, and is in $C^\infty_0(B_2(0))$.

Define
\begin{equation}\label{fai}
\Phi_{r_k}(x):=\Phi(\frac{x-x^k}{r_k}).
\end{equation}
For any $x\in B_{{2r_k}}(x^k)\setminus{B_{r_k}(x^k)},$ we can take $\varepsilon_k>0$ such that
\begin{equation}\label{z1}
u(x^k)+\varepsilon_k\Phi_{r_k}(x^k)\geq A+\varepsilon_k\Phi_{r_k}(x^k+r_ke)\geq u(x)+\varepsilon_k\Phi_{r_k}(x).
\end{equation}
where $e$ is any unit vector in $\mathbb{R}^n$.

Therefore, there exists $\bar{x}^k\in B_{r_k}(x^k)$ such that
\begin{equation}\label{z2}
u(\bar{x}^k)+\varepsilon_k\Phi_{r_k}(\bar{x}^k)=\max_{x\in B_{2r_k}(x^k)}[u(x)+\varepsilon_k\Phi_{r_k}(x)].
\end{equation}
As a consequence,
$$u(\bar{x}^k)+\varepsilon_k\Phi_{r_k}(\bar{x}^k)\geq u(x^k)+\varepsilon_k\Phi_{r_k}(x^k),$$
which implies
$$u(\bar{x}^k)\geq u(x^k)+\varepsilon_k\Phi_{r_k}(x^k)-\varepsilon_k\Phi_{r_k}(\bar{x}^k)\geq u(x^k).$$

It follows from \eqref{key} that
\begin{equation}\label{zzz}
u(\bar{x}^k)\geq \gamma_k A.
\end{equation}
From \eqref{z1} and \eqref{z2}, we deduce that

\begin{equation}\label{z33}
u(\bar{x}^k)+\varepsilon_k\Phi_{r_k}(\bar{x}^k)\geq A\geq u(x),\ \forall x\in \mathbb{R}^n.
\end{equation}
Hence $\bar{x}^k$ is a maximum of the function $u+\varepsilon_k\Phi_{r_k}$ in $\mathbb{R}^n$.

Let $G(t)=|t|^{p-2}t$, we calculate
\begin{equation}\label{iiii}\begin{split}
&(-\Delta)_p^s u(\bar{x}^k)+\frac{\varepsilon_k^{p-1}}{r_k^{sp}}[(-\Delta)_p^s \Phi](\frac{\bar{x}^k-x^k}{r_k})\\
&=(-\Delta)_p^s u(\bar{x}^k)+(-\Delta)_p^s [\varepsilon_k\Phi_{r_k}(\bar{x}^k)]\\
&=C_{n,s,p} P.V.\int_{\mathbb{R}^n}\frac{G(u(\bar{x}^k)-u(y))+G(\varepsilon_k\Phi_{r_k}(\bar{x}^k)-\varepsilon_k\Phi_{r_k}(y))}{|\bar{x}^k-y|^{n+sp}}dy\\
&=C_{n,s,p} P.V.\int_{B_{2r_k}(x^k)}\frac{G(u(\bar{x}^k)-u(y))+G(\varepsilon_k\Phi_{r_k}(\bar{x}^k)-\varepsilon_k\Phi_{r_k}(y))}{|\bar{x}^k-y|^{n+sp}}dy\\
&\ \ \ +C_{n,s,p} \int_{\mathbb{R}^n\setminus{B_{2r_k}(x^k)}}\frac{G(u(\bar{x}^k)-u(y))+G(\varepsilon_k\Phi_{r_k}(\bar{x}^k))}{|\bar{x}^k-y|^{n+sp}}dy\\
&=I_1+I_2.
\end{split}\end{equation}
For $I_1$, we first notice that
$$G(u(\bar{x}^k)-u(y))+G(\varepsilon_k\Phi_{r_k}(\bar{x}^k)-\varepsilon_k\Phi_{r_k}(y))\geq 0$$
due to the strict monotonicity of $G$ and the fact
$$u(\bar{x}^k)+\varepsilon_k\Phi_{r_k}(\bar{x}^k)-u(y)-\varepsilon_k\Phi_{r_k}(y)\geq0,$$
for any $y\in B_{2r_k}(x^k)$. Thus
\begin{equation}\label{ccc}
I_1\geq0.
\end{equation}
Now we estimate $I_2$, it follows from Lemma \ref{appendix} in Appendix and \eqref{z33} that
\begin{equation}\label{ddd}\begin{split}
I_2&= C_{n,s,p}\int_{\mathbb{R}^n\setminus{B_{2r_k}(x^k)}}\frac{G(u(\bar{x}^k)-u(y))+G(\varepsilon_k\Phi_{r_k}(\bar{x}^k))}{|\bar{x}^k-y|^{n+sp}}dy\\
&\geq 2^{2-p}C_{n,s,p}\int_{\mathbb{R}^n\setminus{B_{2r_k}(x^k)}}\frac{G[u(\bar{x}^k)+\varepsilon_k\Phi_{r_k}(\bar{x}^k)-u(y)]}{|\bar{x}^k-y|^{n+sp}}dy\\
&\geq 2^{2-p}C_{n,s,p}\int_{\big(\mathbb{R}^n\setminus{B_{2r_k}(x^k)}\big)\cap\Sigma}\frac{G[u(\bar{x}^k)+\varepsilon_k\Phi_{r_k}(\bar{x}^k)-u(y)]}
{|\bar{x}^k-y|^{n+sp}}dy\\
&\geq A^{p-1}2^{2-p}C_{n,s,p}\int_{\big(\mathbb{R}^n\setminus{B_{2r_k}(x^k)}\big)\cap\Sigma}\frac{1}{|\bar{x}^k-y|^{n+sp}}dy\\
&\geq c_1\int_{\Sigma\setminus{B_{2r_k}(x^k)}}\frac{1}{|x^k-y|^{n+sp}}dy,
\end{split}\end{equation}
where the last inequality we have used the fact
\begin{equation*}
|\bar{x}^k-y|\leq |\bar{x}^k-x^k|+|x^k-y|
\leq \frac{3}{2}|x^k-y|.
\end{equation*}
We choose $r_k=dist(x^k,\partial\Sigma)$, by \eqref{jixian}, there exists $j_0\geq 1$ such that
\begin{equation}\label{dddff}\begin{split}
I_2&\geq c_1\int_{\Sigma\setminus{B_{2r_k}(x^k)}}\frac{1}{|x^k-y|^{n+sp}}dy\\
&\geq c_1\sum_{j=j_0}^{\infty}\frac{\big|\big(B_{2^{j+1}r_k}(x^k)\setminus{B_{2^{j}r_k}(x^k)}\big)\cap\Sigma\big|}{(2^{j+1}r_k)^{n+sp}}\\
&\geq c' \sum_{j=j_0}^{\infty} \frac{1}{(2^{j+1}r_k)^{sp}}= \frac{2^{-j_0}c'}{r_k^{sp}}.
\end{split}\end{equation}
where $c'>0$ depending on $c_0$ and $c_1$.

On the other hand, by \eqref{unbdd} and \eqref{zzz}, we deduce that
$$(-\Delta)_p^su(\bar{x}^k)\leq0,$$
which combining with \eqref{iiii}, \eqref{ccc} and \eqref{dddff}, yields
\begin{equation*}\begin{split}
&\frac{\varepsilon_k^{p-1}}{r_k^{sp}}[(-\Delta)_p^s \Phi](\frac{\bar{x}^k-x^k}{r_k})\\
&\geq(-\Delta)_p^s u(\bar{x}^k)+\frac{\varepsilon_k^{p-1}}{r_k^{sp}}[(-\Delta)_p^s \Phi](\frac{\bar{x}^k-x^k}{r_k})\\
&=(-\Delta)_p^s u(\bar{x}^k)+(-\Delta)_p^s [\varepsilon_k\Phi_{r_k}(\bar{x}^k)]\\
&\geq \frac{2^{-j_0}c'}{r_k^{sp}}.
\end{split}\end{equation*}
In fact, it is easy to check $\big|[(-\Delta)_p^s \Phi](\frac{\bar{x}^k-x^k}{r_k})\big|\leq c$ for $p>2$ (see Lemma 5.2 in Chen and Li \cite{CL2}).
Then we arrive at
\begin{equation}\label{contradiction}
\varepsilon_k^{p-1}\geq 2^{-j_0} c_1.
\end{equation}

Recalling \eqref{key} and \eqref{z1}, we can take $\varepsilon_k$ sufficiently small provided $\gamma_k$ is close to 1 to derive a contradiction with \eqref{contradiction}, and thus complete the proof of Theorem \ref{unbonudedthma}.
\end{proof}

\medskip

To prove Theorem \ref{mainthm11aa}, we also need the following \emph{strong maximum principle}.
\begin{lem}\label{Strongmax}(Strong maximum principle)
Let $D$ be an open set in $\mathbb{R}^n$, possibly unbounded and disconnected. Assume that both $u$ and $v$ are continuous functions
in $C_{loc}^{1,1}(D)\cap \mathcal{L}_{sp}$
and satisfy
\begin{equation}\label{Strong}\begin{cases}
(-\Delta)_p^s u(x)-(-\Delta)_p^s v(x)=f(u(x))-f(v(x)), \ x\in D,\\
u(x)\geq v(x),\ \ \  \ \ \ \ \ \ \ \ \ \ \ \ \ \ \ \ \ \ \ x\in \mathbb{R}^n.
\end{cases}\end{equation}
where $f$ is a continuous function.

Then either $u(x)> v(x)$, or $u(x)\equiv v(x)$ in $\mathbb{R}^n$.
\end{lem}

\begin{proof}
Let
$$w(x)=u(x)-v(x).$$
Assume that there exists $x^0$ in $\mathbb{R}^n$ such that
$$w(x^0)=\min_{x\in \mathbb{R}^n}w(x)=0.$$
It follows from \eqref{Strong} that
\begin{equation}\label{Strongs}\begin{split}
&(-\Delta)_p^s u(x^0)-(-\Delta)_p^s v(x^0)\\
&=C_{n,s,p} P.V.\int_{\mathbb{R}^n}\frac{G(u(x^0)-u(y))-G(v(x^0)-v(y))}{|x^0-y|^{n+sp}}dy.\\
\end{split}\end{equation}
Since
$$[u(x^0)-u(y)]-[v(x^0)-v(y)]=w(x^0)-w(y)=-w(y)\leq0,$$
and due to the monotonicity of $G$, we have

\begin{equation}\label{Strongstr}
G(u(x^0)-u(y))-G(v(x^0)-v(y))\leq0.
\end{equation}
Therefore, by \eqref{Strong}, \eqref{Strongs} and \eqref{Strongstr}, we must have
$$w(y)=0,\ \ \ \text{for any}\ \ y\in\mathbb{R}^n.$$

This completes the proof of the lemma.
\end{proof}

We consider
\begin{equation}\label{fractionalbianjiewaiaa}\begin{cases}
(-\Delta)_p^s u(x)=f(u(x)),& x\in\Omega,\\
u(x)>0, & x\in\Omega.\\
\end{cases}\end{equation}

Based on the above two \emph{maximum principles}, we prove the following theorem.

\begin{thm}\label{mainthm11aa}
Let $u\in C_{loc}^{1,1}(\Omega)\cap \mathcal{L}_{sp}$ be a bounded solution of \eqref{fractionalbianjiewaiaa} and
$$u(x)<\mu, \;\; x\in \mathbb{R}^n\setminus \Omega.$$

Assume that $f$ is a continuous function satisfying

(a) There exists $\mu>0$ such that $f(t)>0$
on $(0,\mu)$, and $f(t)\leq0$ for $t\geq\mu$.

Then $u<\mu$ in $\Omega$.
\end{thm}

\begin{proof}
Without loss of generality, we always assume that $\mu=1$ in conditions (a)--(c) in the rest of our paper.

Now we first show that $u(x)\leq 1,$ for all $x\in\Omega$.
Indeed, if $u>1$ somewhere, let $D$ be a component of the set where $u>1$. Notice that
$u<1$ in $\mathbb{R}^n\setminus\Omega$, let
$$w_1=u-1.$$
Since $f(1)=0$, $f(u(x))\leq0$, $x\in D$, we have
$$(-\Delta)^s_pw_1(x)\leq 0,\ \ x\in D.$$
It follows from Theorem \ref{unbonudedthma} that
$$w_1(x)\leq 0,\ \ x\in D.$$
Thus $u\leq 1$ in $D$, which contradicts the assumption $u>1$ somewhere.
So we derive that
$$u(x)\leq 1, \ \  x\in\Omega.$$
By \emph{strong maximum principle} (Lemma \ref{Strongmax}), we conclude that $u<1$ in $\Omega$.

This completes the proof of Theorem \ref{mainthm11aa}.
\end{proof}

\medskip

\section{The proof of Theorem \ref{mainthm11}}
In this section, we consider
\begin{equation}\label{fractionalbianjiewai3}\begin{cases}
(-\Delta)_p^s u(x)=f(u(x)),&x\in\Omega,\\
u(x)>0, &x\in\Omega.\\
\end{cases}\end{equation}

We prove
\begin{thm}\label{mainthm1133}
Let $u\in C_{loc}^{1,1}(\Omega)\cap \mathcal{L}_{sp}$ be a bounded solution of \eqref{fractionalbianjiewai3} with
$$0\leq u(x)<\mu, \;\; x\in \mathbb{R}^n\setminus \Omega.$$

Assume that $f$ is  continuous and satisfy condition (a),
and for some $0<t_0<t_1<\mu$,

(b) $f(t)\geq\delta_0t$ on $[0,t_0]$ for some $\delta_0>0$, and

(c) $f(t)$ is nonincreasing on $(t_1,\mu)$.

Then $u(x)\rightarrow\mu$ uniformly in $\Omega$ as $dist(x,\partial\Omega)\rightarrow\infty$.\\
\end{thm}

We first prove that the bounded solution of  \eqref{fractionalbianjiewai3} is bounded away from zero at points far away from the boundary.
\begin{lem}\label{leme}
There exist $\varepsilon_1,$ $R_0>0$ with $R_0$ depending only on $n$ and $\delta_0$
(recall condition (b)) such that
$$u(x)>\varepsilon_1 \ \ \ \text{if}\ \ \ dist(x,\partial\Omega)>R_0.$$
\end{lem}

\medskip
\begin{proof}

Let $\lambda_1=\lambda_1(B_1(0))$ be the principle eigenvalue of $(-\Delta)_p^s $ in $B_1(0)$ with
Dirichlet boundary condition, assume that $\psi$ be the eigenfunction of
$(-\Delta)_p^s $ in $B_1(0)$, i.e.,

\begin{equation*}\begin{cases}
&(-\Delta)_p^s\psi(x)=\lambda_1\psi(x),\ \psi(x)>0,\ x\in B_1(0),\\
&\psi(x)=0,\ \ \ \ \ \ \ \ \ \ \ x\in \mathbb{R}^n\setminus{B_1(0)},
\end{cases}\end{equation*}
with $\max_{x\in B_1}\psi(x)=\psi(0)=1.$

Define
$$\psi_R(x)=\psi(\frac{x}{R}),\ \  \psi_{\varepsilon,R}(x)=\varepsilon\psi_R(x).$$
Then, it is obvious that $\psi_{\varepsilon,R}(0)=\varepsilon\psi(0)=\varepsilon.$ For $\varepsilon\in(0,t_0]$, there exists $R_0$
sufficiently large such that $\frac{\lambda_1}{R_0^{sp}}<\delta_0$. For simplicity, we use $R$ instead of $R_0$, then
\begin{equation*}\begin{split}
(-\Delta)_p^s\psi_{\varepsilon,R}(x)&=\frac{\varepsilon}{R^{sp}}[(-\Delta)_p^s\psi](\frac{x}{R})\\
&=\frac{\lambda_1\varepsilon}{R^{sp}}\psi_R(x)\\
&\leq \delta_0(\varepsilon\psi_R(x))\\
&\leq f(\psi_{\varepsilon,R}(x)),
\end{split}\end{equation*}
where the last inequality is due to the condition $f(t)\geq \delta_0t$ for some $t\in [0,t_0]$.

It follows from \eqref{fractionalbianjiewai3} that
\begin{equation}\label{fffaa}\begin{split}
(-\Delta)_p^su(x)-(-\Delta)_p^s \psi_{\varepsilon,R}(x)&=f(u(x))-\frac{\lambda_1}{R^{sp}} \psi_{\varepsilon,R}(x)\\
&\geq f(u(x))-f(\psi_{\varepsilon,R}(x)).
\end{split}\end{equation}
For $y^0\in\Omega$ with $dist(y^0,\partial\Omega)>R$, we choose $\varepsilon_0$ small enough such that
$$\varepsilon_0<\inf_{x\in B_R(y^0)}u(x).$$
Then, set $\varepsilon_1=\min\{\varepsilon_0,t_0\}$, we have
\begin{equation}\label{fffff}
u(x)>\varepsilon_1\psi_R(x-y^0)=\psi_{\varepsilon_1,R}(x-y^0),\ \ x\in \overline{B_R(y^0)}.
\end{equation}
For $t\in[0,1]$ and $y\in\Omega$ with $dist(y,\partial\Omega)>R$,  let $y_t=ty+(1-t)y^0$ and
$$w_t(x)=u(x)-\psi_{\varepsilon_1,R}(x-y_t),\ \ x\in \overline{B_R(y_t)}.$$
It follows from \eqref{fffff} that
$$w_0(x)>0,\ \ x\in \overline{B_R(y^0)},$$
and
\begin{equation}\label{partial}
w_t(x)>0,\ \ x\in \partial B_R(y_t).
\end{equation}
Now we will prove that
\begin{equation}\label{gggd}
w_t(x)>0,\ \ \text{for any}\ \ x\in B_R(y_t).
\end{equation}
Suppose on the contrary that there is a fist $t$ such that the graph of $\psi_{\varepsilon_1,R}(\cdot-y_t)$ touches that of $u$
at some point $\bar{x}_R\in \overline{B_R(y_t)}$. Then, from \eqref{partial}, we deduce that $\bar{x}_R\in B_R(y_t)$ and
\begin{equation}\label{abc}
w_t(\bar{x}_R)=0.
\end{equation}
On the other hand,
\begin{equation*}\begin{split}
&(-\Delta)_p^su(\bar{x}_R)-(-\Delta)_p^s \psi_{\varepsilon_1,R}(\bar{x}_R-y_t)\\
&=C_{n,s,p} P.V.\int_{\mathbb{R}^n}\frac{G(u(\bar{x}_R)-u(z))-G(\psi_{\varepsilon_1,R}(\bar{x}_R-y_t)-\psi_{\varepsilon_1,R}(z-y_t))}{|\bar{x}_R-z|^{n+sp}}dz\\
&=C_{n,s,p} P.V.\int_{B_R(y_t)}\frac{G(u(\bar{x}_R)-u(z))-G(\psi_{\varepsilon_1,R}(\bar{x}_R-y_t)-\psi_{\varepsilon_1,R}(z-y_t))}{|\bar{x}_R-z|^{n+sp}}dz\\
&\ \  \ +C_{n,s,p} \int_{\mathbb{R}^n\setminus {B_R(y_t)}}\frac{G(u(\bar{x}_R)-u(z))-G(\psi_{\varepsilon_1,R}(\bar{x}_R-y_t))}{|\bar{x}_R-z|^{n+sp}}dz\\
&=C_{n,s,p}\{I_1+I_2\}.
\end{split}\end{equation*}

We first estimate $I_1$, for $z\in B_R(x_R)$, we have
$$G(u(\bar{x}_R)-u(z))-G(\psi_{\varepsilon_1,R}(\bar{x}_R-y_t)-\psi_{\varepsilon_1,R}(z-y_t))\leq 0\ \ \text{but}\ \  \not \equiv0,$$
due to the monotonicity of $G$ and the fact that
$$[u(\bar{x}_R)-u(z)]-[\psi_{\varepsilon_1,R}(\bar{x}_R-y_t)-\psi_{\varepsilon_1,R}(z-y_t))]=w_t(\bar{x}_R)-w_t(z)\leq0\ \ \text{but}\ \  \not \equiv0.$$
One immediately has
$$I_1<0.$$
For $I_2$, $z\in \mathbb{R}^n\setminus{B_R(x_R)}$, we also can deduce
\begin{equation*}
G(u(\bar{x}_R)-u(z))-G(\psi_{\varepsilon_1,R}(\bar{x}_R-y_t))\leq 0\ \ \text{but}\ \  \not \equiv0.
\end{equation*}
Thus $I_2<0$, it follows that
\begin{equation}\label{abcd}
(-\Delta)_p^su(\bar{x}_R)-(-\Delta)_p^s \psi_{\varepsilon_1,R}(\bar{x}_R-y_t)<0.
\end{equation}

On the other hand, by \eqref{fffaa}, we obtain
\begin{equation}\label{abcde}\begin{split}
(-\Delta)_p^su(\bar{x}_R)-(-\Delta)_p^s \psi_{\varepsilon_1,R}(\bar{x}_R-y_t)&\geq f(u(\bar{x}_R))-f(\psi_{\varepsilon_1,R}(\bar{x}_R-y_t))\\
&=0.
\end{split}\end{equation}
It follows from \eqref{abc}, \eqref{abcd} and \eqref{abcde} that \eqref{gggd} must be valid. Let $t=1$, we obtain
\begin{equation*}
u(x)>\varepsilon_1\psi_R(x-y),\ \ \text{for any}\ \ x\in \overline{B_R(y)}.
\end{equation*}
In particular, $x=y$, it yields that
$$u(y)>\varepsilon_1,\ \ \text{for all}\ y\in\Omega\ \ \text{with}\ \ dist(y,\partial\Omega)>R.$$
This completes the proof of the lemma.
\end{proof}

\medskip
Now, we prove Theorem \ref{mainthm1133}.
Let
\begin{equation*}
\phi(x)=\begin{cases}
c_ne^{\frac{1}{|x|^2-1}}, &|x|<1,\\
0,&|x|\geq1.
\end{cases}
\end{equation*}
We choose $c_n$ such that $\phi(0)=1$. Set
$$\phi_R(x)=\phi(\frac{x-x_R}{R}),$$
where $x_R$ satisfies $dist(x_R,\partial\Omega)>2R$ and $B_R(x_R)\subset \Omega$.

It is obvious that $\phi_R(x_R)=\max_{x\in B_R(x_R)}\phi_R(x)=\phi(0)=1$, and $\phi_R$ satisfies
\begin{equation*}\begin{cases}
(-\Delta)_p^s \phi_R(x)\leq\frac{C}{R^{sp}}, &x\in B_R(x_R),\\
\phi_R(x)=0,& x\in \mathbb{R}^n\setminus B_R(x_R).
\end{cases}\end{equation*}

As a consequence, it follows from \eqref{fractionalbianjiewai3} that
\begin{equation}\label{important}
(-\Delta)_p^su(x)-(-\Delta)_p^s \phi_R(x)\geq f(u(x))-\frac{C}{R^{sp}}.
\end{equation}

Let $w_R(x)=u(x)-\phi_R(x)$, since $\phi_R(x_R)=1$ and notice that $u(x)<1$, $x\in B_R(x_R)$. We infer that
there exists $\bar{x}_R$ such that
$$w_R(\bar{x}_R)=\min_{x\in B_R(x_R)}w_R(x)<0,$$
which implies
$$u(\bar{x}_R)-\phi_R(\bar{x}_R)\leq u(x_R)-\phi_R(x_R).$$
It follows immediately that
\begin{equation}\label{biger}
u(\bar{x}_R)\leq u(x_R)-(\phi_R(x_R)-\phi_R(\bar{x}_R))\leq u(x_R).
\end{equation}
On the other hand,
\begin{equation*}\begin{split}
&(-\Delta)_p^su(\bar{x}_R)-(-\Delta)_p^s \phi_R(\bar{x}_R)\\
&=C_{n,s,p} P.V.\int_{\mathbb{R}^n}\frac{G(u(\bar{x}_R)-u(y))-G(\phi_R(\bar{x}_R)-\phi_R(y))}{|\bar{x}_R-y|^{n+sp}}dy\\
&=C_{n,s,p} P.V.\int_{B_R(x_R)}\frac{G(u(\bar{x}_R)-u(y))-G(\phi_R(\bar{x}_R)-\phi_R(y))}{|\bar{x}_R-y|^{n+sp}}dy\\
&\ \  \ +C_{n,s,p} \int_{\mathbb{R}^n\setminus {B_R(x_R)}}\frac{G(u(\bar{x}_R)-u(y))-G(\phi_R(\bar{x}_R))}{|\bar{x}_R-y|^{n+sp}}dy\\
&=C_{n,s,p}\{I_1+I_2\}.
\end{split}\end{equation*}

For $I_1$, $y\in B_R(x_R)$, we have
$$G(u(\bar{x}_R)-u(y))-G(\phi_R(\bar{x}_R)-\phi_R(y))\leq 0\ \ \text{but}\ \  \not \equiv0,$$
due to the monotonicity of $G$ and the fact that
$$[u(\bar{x}_R)-u(y)]-[\phi_R(\bar{x}_R)-\phi_R(y))]=w(\bar{x}_R)-w(y)\leq0\ \ \text{but}\ \  \not \equiv0.$$
One immediately has
$$I_1<0.$$
For $I_2$, $y\in \mathbb{R}^n\setminus{B_R(x_R)}$, we also can deduce
$$G(u(\bar{x}_R)-u(y))-G(\phi_R(\bar{x}_R))\leq 0\ \ \text{but}\ \  \not \equiv0.$$
Thus $I_2<0$, it follows that
$$(-\Delta)_p^su(\bar{x}_R)-(-\Delta)_p^s \phi_R(\bar{x}_R)<0,$$
combining this with \eqref{important} gives that
$$f(u(\bar{x}_R))<\frac{C}{R^{sp}}.$$
Hence, we obtain that
\begin{equation}\label{mm}
f(u(\bar{x}_R))\rightarrow0, \ \text{as}\ \  R\rightarrow\infty.
\end{equation}
Now we claim that
\begin{equation}\label{limit}
u(\bar{x}_R)\rightarrow1,\ \ \text{as} \ \ R\rightarrow\infty.
\end{equation}
In fact, by Lemma \ref{leme}, we have $u(\bar{x}_R)\geq \varepsilon_1$ ($\varepsilon_1>0$) as $\bar{x}_R$ away from boundary,
assume that $u(\bar{x}_R)\in [0,t_0]$, we get from condition (b),
\begin{equation}\label{mmm}
f(u(\bar{x}_R))\geq \delta_0u(\bar{x}_R)\geq\delta_0\varepsilon_1
\end{equation}
for some $\delta_0>0$.
Meanwhile, by condition (a), $f$ is a continuous function in $\mathbb{R}$ and $f(t)>0$ in $(0,1)$, we have
\begin{equation}\label{mmmmk}
\inf_{t\in [t_0,t_1]}f(t)=c_0>0.
\end{equation}

Therefore, we derive from \eqref{mm}, \eqref{mmm} and \eqref{mmmmk} that $u(\bar{x}_R)$ must fall in open interval
$(t_1,1)$, in  which  $f(t)$ is nonincreasing due to condition (c). Hence \eqref{limit} must be valid.
It follows from \eqref{biger} that
$$1>u(x_R)\geq u(\bar{x}_R)\rightarrow1, \ \ \text{as}\ \  R\rightarrow\infty.$$
Thus, we obtain that $u(x_R)\rightarrow1$ as $R\rightarrow\infty$.

This completes the proof of Theorem \ref{mainthm1133}.

\section{The proof of Theorem \ref{mainthm}}

In \cite{Li}, Li considered the following equation,
\begin{equation}\label{fractionalhold}\begin{cases}
(-\Delta)_p^s u(x)=g(x,u),\ u>0,& x\in\Gamma,\\
u(x)\leq0, &x\in \mathbb{R}^n\setminus \Gamma.
\end{cases}\end{equation}
\medskip

Based on Jin and Li \cite{JL} (the \emph{boundary H\"{o}lder regularity}) and  Brasco, Lindgren and Schikorra \cite{BLS} (the \emph{interior H\"{o}lder regularity}),
Li \cite{Li} obtained the following \emph{uniform H\"{o}lder norm estimate} in $\mathbb{R}^n$ for the fractional p-Laplacian.
\begin{lem}\label{holderestimate}(The uniform H\"{o}lder norm estimate)
Assume that $\Gamma$ is a domain (possibly unbounded) with the uniform two-sided ball condition,
and $u\in C_{loc}^{1,1}\cap \mathcal{L}_{sp}$
is a bounded solution of \eqref{fractionalhold}.
If $g(x,u)$ is bounded, then there exists
$\alpha\in (0,s)$ such that $u\in C^\alpha(\mathbb{R}^n)$. Moreover,
$$[u]_{C^\alpha(\mathbb{R}^n)}\leq C\big(1+\|u\|_{L^\infty(\Gamma)}+\|g\|^{\frac{1}{p-1}}_{L^\infty(\Gamma)}\big),$$
where $C$ is a constant depending on $\alpha$, $s$, $p$, $\Gamma$.
\end{lem}

Based on this uniform estimate, we apply the sliding method to derive the monotonicity and uniqueness of solutions for
\begin{equation}\label{fractionalww}\begin{cases}
(-\Delta)_p^s u(x)=f(u(x)),&x\in\Omega,\\
u(x)>0,&x\in\Omega,\\
u(x)=0,&x\in \mathbb{R}^n\setminus \Omega,
\end{cases}\end{equation}
where $\Omega$ satisfies the uniform two-sided ball condition.

\medskip

\begin{thm}\label{mainthmaaa}
Let $u\in C_{loc}^{1,1}(\Omega) \cap \mathcal{L}_{sp}$ be a bounded solution of \eqref{fractionalww}.
Assume that $f$ is a continuous function and satisfies conditions (a)-(c) for some $0<t_0<t_1<\mu$.

Then $u$ is strictly monotone increasing in $x_n$.

Furthermore the bounded solution of \eqref{fractionalww} is unique.
\end{thm}

\medskip

\begin{proof}
We will carry out the proof of Theorem \ref{mainthmaaa} in three steps.

For $\tau\geq0$, denote
$$u_\tau(x):=u(x+\tau e_n),\ \ w_\tau(x):=u(x)-u_\tau(x),$$
where $e_n=(0,0,\cdots,1)$.

In step 1, we will show that for $\tau$ sufficiently large, we
have
\begin{equation}\label{qidiansliding}
w_\tau(x)\leq0,\ \ x\in\mathbb{R}^n.
\end{equation}
This provides the starting point for the sliding method. Then in step 2, we decrease $\tau$ continuously as long
as \eqref{qidiansliding} holds to its limiting position. Define
$$\tau_0:=\inf\{\tau>0\mid \ w_\tau(x)\leq0,\  \forall x\in\mathbb{R}^n\}.$$
We will show that $\tau_0=0$. Then we deduce that
the solution $u$ must be strictly monotone increasing in $x_n$.
In step 3, we will prove the uniqueness by constructing
the sub-solution.
\medskip

We now show the details in the three steps.

\medskip
\emph{Step 1.} We show that for $\tau$ sufficiently large, we have
\begin{equation}\label{tao}
w_\tau(x)\leq0,\ \  x\in\mathbb{R}^n.
\end{equation}

For $h>0$, define
$$\Omega_h:=\{x\in\mathbb{R}^n\mid \varphi(x')<x_n<\varphi(x')+h\}.$$
By Theorem \ref{mainthm1133}, there exists an $M_0>0$ large enough such that for $\tau\geq M_0$,
$$u_\tau (x) \in(t_1,1),\ \forall x\in \Omega_h.$$
Suppose \eqref{tao} is violated, there exists a constant $A>0$ such that
$$\sup_{x\in \mathbb{R}^n}w_\tau(x)= A,$$
hence there exists a sequence $\{x^k\}$ in $\mathbb{R}^n$ such that
$$w_\tau(x^k)\rightarrow A,\ \text{as}\ k\rightarrow\infty.$$
Since $u=0$ in $\mathbb{R}^n\setminus{\overline{\Omega}}$, it yields that
$$w_\tau(x)\leq 0,\ \ \ \forall x\in\mathbb{R}^n\setminus{\overline{\Omega}}.$$
Moreover, thanks to Theorem \ref{mainthm1133}, there exists an $M>M_0>0$ such that the sequence $x^k\in\Omega_M$.
Similar to the argument as Theorem \ref{unbonudedthma}, there exists $\bar{x}^k\in B_1(x^k)$ such that
\begin{equation}\label{lastin}
w_\tau(\bar{x}^k)+\varepsilon_k\Phi_{1}(\bar{x}^k)=\max_{x\in B_1(x^k)}[w_\tau(x)+\varepsilon_k\Phi_{1}(x)]
=\max_{x\in\mathbb{R}^n}[w_\tau(x)+\varepsilon_k\Phi_{1}(x)]>A,
\end{equation}
where the definition of $\Phi_{r_k}$ is the same as \eqref{fai} with $r_k=1$.

It follows from Lemma \ref{appendix2} in Appendix
and the monotonicity of $f(t)$ for $t\in(t_1,1)$ that
\begin{equation}\label{qqffhhh}\begin{split}
&(-\Delta)_p^s (u+\varepsilon_k\Phi_{1})(\bar{x}^k)-(-\Delta)_p^s u_\tau(\bar{x}^k)\\
&=(-\Delta)_p^s (u+\varepsilon_k\Phi_{1})(\bar{x}^k)-(-\Delta)_p^s u(\bar{x}^k)+(-\Delta)_p^s u(\bar{x}^k)-(-\Delta)_p^s u_\tau(\bar{x}^k)\\
&\leq f(u(\bar{x}^k))-f(u_\tau(\bar{x}^k))+\varepsilon_k C_\delta+C\delta^{p(1-s)}\\
&\leq \varepsilon_k C_\delta+C\delta^{p(1-s)}.
\end{split}\end{equation}

On the other hand, we calculate
\begin{equation}\label{qqiiii}\begin{split}
&(-\Delta)_p^s (u+\varepsilon_k\Phi_{1})(\bar{x}^k)-(-\Delta)_p^s u_\tau(\bar{x}^k)\\
&=C_{n,s,p} P.V.\int_{\mathbb{R}^n}\frac{G(u(\bar{x}^k)+\varepsilon_k\Phi_{1}(\bar{x}^k)-u(y)-\varepsilon_k\Phi_{1}(y))-
G(u_\tau(\bar{x}^k)-u_\tau(y))}{|\bar{x}^k-y|^{n+sp}}dy\\
&=C_{n,s,p} P.V.\int_{B_{2}(x^k)}\frac{G(u(\bar{x}^k)+\varepsilon_k\Phi_{1}(\bar{x}^k)-u(y)-\varepsilon_k\Phi_{1}(y))-
G(u_\tau(\bar{x}^k)-u_\tau(y))}{|\bar{x}^k-y|^{n+sp}}dy\\
&\ \ \ +C_{n,s,p} \int_{\mathbb{R}^n\setminus{B_{2}(x^k)}}\frac{G(u(\bar{x}^k)+\varepsilon_k\Phi_{1}(\bar{x}^k)-u(y)-\varepsilon_k\Phi_{1}(y))-
G(u_\tau(\bar{x}^k)-u_\tau(y))}{|\bar{x}^k-y|^{n+sp}}dy\\
&=I_1+I_2.
\end{split}\end{equation}
For $I_1$, we first notice that
$$G(u(\bar{x}^k)+\varepsilon_k\Phi_{1}(\bar{x}^k)-u(y)-\varepsilon_k\Phi_{1}(y))-
G(u_\tau(\bar{x}^k)-u_\tau(y))\geq 0$$
due to the strict monotonicity of $G$ and the fact
\begin{equation*}\begin{split}
&u(\bar{x}^k)+\varepsilon_k\Phi_{1}(\bar{x}^k)-u(y)-\varepsilon_k\Phi_{1}(y)-
(u_\tau(\bar{x}^k)-u_\tau(y))\\
&\ \ =w_\tau(\bar{x}^k)+\varepsilon_k\Phi_{1}(\bar{x}^k)-(w_\tau(y)+\varepsilon_k\Phi_{1}(y))\geq 0,
\end{split}\end{equation*}
for any $y\in B_{2}(x^k)$. Thus
\begin{equation}\label{qccc}
I_1\geq0.
\end{equation}
Now we estimate $I_2$, one can infer from Lemma \ref{appendix} in Appendix and \eqref{lastin} that
\begin{equation*}\label{qddd}\begin{split}
I_2&= C_{n,s,p}\int_{\mathbb{R}^n\setminus{B_{2}(x^k)}}\frac{G(u(\bar{x}^k)+\varepsilon_k\Phi_{1}(\bar{x}^k)-u(y)-\varepsilon_k\Phi_{1}(y))-
G(u_\tau(\bar{x}^k)-u_\tau(y))}{|\bar{x}^k-y|^{n+sp}}dy\\
&\geq 2^{2-p}C_{n,s,p}\int_{\mathbb{R}^n\setminus{B_{2}(x^k)}}\frac{G[w_\tau(\bar{x}^k)+\varepsilon_k
\Phi_{1}(\bar{x}^k)-u(y)]}{|\bar{x}^k-y|^{n+sp}}dy\\
&\geq 2^{2-p}C_{n,s,p}\int_{\big(\mathbb{R}^n\setminus{B_{2}(x^k)}\big)\cap{\big(\mathbb{R}^n\setminus{\overline{\Omega}}\big)}
}\frac{G[w_\tau(\bar{x}^k)+\varepsilon_k\Phi_{1}(\bar{x}^k)-w_\tau(y)]}
{|\bar{x}^k-y|^{n+sp}}dy\\
&\geq A^{p-1}2^{2-p}C_{n,s,p}\int_{\big(\mathbb{R}^n\setminus{B_{2}(x^k)}\big)\cap{\big(\mathbb{R}^n\setminus{\overline{\Omega}}\big)}}\frac{1}{|\bar{x}^k-y|^{n+sp}}dy\\
&\geq c_1\int_{\big(\mathbb{R}^n\setminus{B_{2}(x^k)}\big)\cap{\big(\mathbb{R}^n\setminus{\overline{\Omega}}\big)}}\frac{1}{|x^k-y|^{n+sp}}dy,
\end{split}\end{equation*}
where in the last inequality we have used the fact
\begin{equation*}
|\bar{x}^k-y|\leq |\bar{x}^k-x^k|+|x^k-y|
\leq \frac{3}{2}|x^k-y|.
\end{equation*}
As a consequence, we get
\begin{equation}\label{qdddff}
I_2\geq c_1\int_{\big(\mathbb{R}^n\setminus{B_{2}(x^k)}\big)\cap{\big(\mathbb{R}^n\setminus{\overline{\Omega}}\big)}}\frac{1}{|x^k-y|^{n+sp}}dy\geq c_M>0.
\end{equation}
where $x^k\in\Omega_M$.

It follows from \eqref{qqffhhh}, \eqref{qqiiii}, \eqref{qccc}, \eqref{qdddff} that
\begin{equation*}
\varepsilon_k C_\delta+C\delta^{p(1-s)}\geq c_M.
\end{equation*}
Choosing $\delta=\big(\frac{c_M}{2C}\big)^{\frac{1}{p(1-s)}}$, we arrive at
\begin{equation}\label{lailai}
\varepsilon_k C_\delta\geq \frac{c_M}{2}.
\end{equation}
Since the left hand side of \eqref{lailai} must go to zero as $\varepsilon_k\rightarrow0$ ($k\rightarrow\infty$), which contradicts the right hand side of \eqref{lailai}.
Therefore, \eqref{tao} must be true for sufficiently large $\tau$.
\medskip

\emph{Step 2.}
Now we prove that for $\forall \tau>0$,
\begin{equation}\label{udandiao}
w_\tau(x)<0,\ \ \forall x\in\mathbb{R}^n.
\end{equation}

\emph{Step 1} provides a starting point from which we can decrease $\tau$ continuously from $\tau\geq N$ as long
as \eqref{tao} holds, define
$$\tau_0:=\inf\{\tau>0\mid \ w_\tau(x)\leq0,\  \forall x\in\mathbb{R}^n\}.$$
We show that
\begin{equation}\label{taoweiling}
\tau_0=0.
\end{equation}
Suppose on the contrary $\tau_0>0$. By continuity, we see that $w_{\tau_0}(x)\leq0$.
One can infer from \emph{strong maximum principle} (Lemma \ref{Strongmax}) that
$$w_{\tau_0}(x)<0,\  \forall x\in\mathbb{R}^n.$$

Then the following two cases may occur.
\medskip

\emph{Case 1.} Suppose that
$$\sup_{x\in \Omega_M}w_{\tau_0}(x)=0,$$
where $M$ is the same as \emph{Step 1}.

Then there exists a sequence $\{x^k\}$ in $\Omega_M$ such that
\begin{equation}\label{wtaoling}
w_{\tau_0}(x^k)\rightarrow0,\ \text{as}\ k\rightarrow\infty.
\end{equation}
Similar to the argument as Theorem \ref{unbonudedthma}, by \eqref{wtaoling},
then there exists $\bar{x}^k\in B_1(x^k)$ such that
\begin{equation}\label{jixianweizhi}
w_{\tau_0}(\bar{x}^k)+\varepsilon_k\Phi_1(\bar{x}^k)=\max_{x\in B_1(x^k)}[w_{\tau_0}(x)+\varepsilon_k\Phi_1(x)]
=\max_{x\in\mathbb{R}^n}[w_{\tau_0}(x)+\varepsilon_k\Phi_1(x)]>0.
\end{equation}
where the definition of $\Phi_1$ is the same as \eqref{fai} with $r_k=1$.

By \eqref{jixianweizhi}, we can deduce that
$$0>w_{\tau_0}(\bar{x}^k)\geq w_{\tau_0}(x^k)+\varepsilon_k\Phi_1(x^k)-\varepsilon_k\Phi_1(\bar{x}^k)\geq w_{\tau_0}(x^k)\rightarrow0,
\ \text{as}\ k\rightarrow\infty,$$
which implies that
\begin{equation}\label{wwwjixian}
w_{\tau_0}(\bar{x}^k)\rightarrow0,\ \text{as}\ k\rightarrow\infty.
\end{equation}
It follows from Lemma \ref{appendix2} in Appendix that

\begin{equation}\label{zuihou55}\begin{split}
&(-\Delta)_p^s (u+\varepsilon_k\Phi_1)(\bar{x}^k)-(-\Delta)_p^s u_{\tau_0}(\bar{x}^k)\\
&=(-\Delta)_p^s (u+\varepsilon_k\Phi_{1})(\bar{x}^k)-(-\Delta)_p^s u(\bar{x}^k)+(-\Delta)_p^s u(\bar{x}^k)-(-\Delta)_p^s u_{\tau_0}(\bar{x}^k)\\
&\leq f(u(\bar{x}^k))-f(u_{\tau_0}(\bar{x}^k))+\varepsilon_k C_\delta+C\delta^{p(1-s)}.\\
\end{split}\end{equation}

On the other hand, we calculate
\begin{equation}\label{qqiiii55}\begin{split}
&(-\Delta)_p^s (u+\varepsilon_k\Phi_{1})(\bar{x}^k)-(-\Delta)_p^s u_{\tau_0}(\bar{x}^k)\\
&=C_{n,s,p} P.V.\int_{\mathbb{R}^n}\frac{G(u(\bar{x}^k)+\varepsilon_k\Phi_{1}(\bar{x}^k)-u(y)-\varepsilon_k\Phi_{1}(y))-
G(u_{\tau_0}(\bar{x}^k)-u_{\tau_0}(y))}{|\bar{x}^k-y|^{n+sp}}dy\\
&=C_{n,s,p} P.V.\int_{B_{2}(x^k)}\frac{G(u(\bar{x}^k)+\varepsilon_k\Phi_{1}(\bar{x}^k)-u(y)-\varepsilon_k\Phi_{1}(y))-
G(u_{\tau_0}(\bar{x}^k)-u_{\tau_0}(y))}{|\bar{x}^k-y|^{n+sp}}dy\\
&\ \ \ +C_{n,s,p} \int_{\mathbb{R}^n\setminus{B_{2}(x^k)}}\frac{G(u(\bar{x}^k)+\varepsilon_k\Phi_{1}(\bar{x}^k)-u(y)-\varepsilon_k\Phi_{1}(y))-
G(u_{\tau_0}(\bar{x}^k)-u_{\tau_0}(y))}{|\bar{x}^k-y|^{n+sp}}dy\\
&=I_1+I_2.
\end{split}\end{equation}
For $I_1$, we first notice that
$$G(u(\bar{x}^k)+\varepsilon_k\Phi_{1}(\bar{x}^k)-u(y)-\varepsilon_k\Phi_{1}(y))-
G(u_{\tau_0}(\bar{x}^k)-u_{\tau_0}(y))\geq 0$$
due to the strict monotonicity of $G$ and the fact
\begin{equation*}\begin{split}
&u(\bar{x}^k)+\varepsilon_k\Phi_{1}(\bar{x}^k)-u(y)-\varepsilon_k\Phi_{1}(y)-
(u_{\tau_0}(\bar{x}^k)-u_{\tau_0}(y))\\
&\ \ =w_{\tau_0}(\bar{x}^k)+\varepsilon_k\Phi_{1}(\bar{x}^k)-(w_{\tau_0}(y)+\varepsilon_k\Phi_{1}(y))\geq 0,
\end{split}\end{equation*}
for any $y\in B_{2}(x^k)$. Thus
\begin{equation}\label{qccc55}
I_1\geq0.
\end{equation}
Now we estimate $I_2$, one can infer from Lemma \ref{appendix} in Appendix and \eqref{jixianweizhi} that
\begin{equation}\label{qddd55}\begin{split}
I_2&= C_{n,s,p}\int_{\mathbb{R}^n\setminus{B_{2}(x^k)}}\frac{G(u(\bar{x}^k)+\varepsilon_k\Phi_{1}(\bar{x}^k)-u(y)-\varepsilon_k\Phi_{1}(y))-
G(u_{\tau_0}(\bar{x}^k)-u_{\tau_0}(y))}{|\bar{x}^k-y|^{n+sp}}dy\\
&\geq 2^{2-p}C_{n,s,p}\int_{\mathbb{R}^n\setminus{B_{2}(x^k)}}\frac{G[w_{\tau_0}(\bar{x}^k)+\varepsilon_k
\Phi_{1}(\bar{x}^k)-w_{\tau_0}(y)]}{|\bar{x}^k-y|^{n+sp}}dy\\
&\geq c'\int_{\mathbb{R}^n\setminus{B_{2}(x^k)}}\frac{G[-w_{\tau_0}(y)]}
{|x^k-y|^{n+sp}}dy,\\
\end{split}\end{equation}
where for the last inequality we have used the fact
\begin{equation*}
|\bar{x}^k-y|\leq |\bar{x}^k-x^k|+|x^k-y|
\leq \frac{3}{2}|x^k-y|.
\end{equation*}

It follows from \eqref{qqiiii55},\eqref{qccc55} and \eqref{qddd55} that
\begin{equation}\label{vvgg55}\begin{split}
&(-\Delta)_p^s (u+\varepsilon_k\Phi_{1})(\bar{x}^k)-(-\Delta)_p^s u_{\tau_0}(\bar{x}^k)\\
&\geq c'\int_{\mathbb{R}^n\setminus{B_{2}(x^k)}}\frac{G[-w_{\tau_0}(y)]}
{|x^k-y|^{n+sp}}dy\\
&= c'\int_{\mathbb{R}^n\setminus{B_{2}(0)}}\frac{G[-w_{\tau_0}(y+x^k)]}
{|y|^{n+sp}}dy,
\end{split}\end{equation}
where $x^k\in\Omega_M$.

As a consequence, by \eqref{zuihou55}, we have
\begin{equation}\label{shouliandao}
f(u(\bar{x}^k))-f(u_{\tau_0}(\bar{x}^k))+\varepsilon_k C_\delta+C\delta^{p(1-s)}
\geq c'\int_{\mathbb{R}^n\setminus{B_{2}(0)}}\frac{G[-w_{\tau_0}(y+x^k)]}
{|y|^{n+sp}}dy.
\end{equation}

Let
$$w_{\tau_0}^k(y)=u(y+x^k)-u_{\tau_0}(y+x^k).$$

From Lemma \ref{holderestimate}, we know that $u$ is uniformly H\"{o}lder continuous in $\mathbb{R}^n$,
hence, $w_{\tau_0}^k$ is equi-continuous in $\mathbb{R}^n$. By the \emph{Arzel\`{a}-Ascoli theorem},
there exists $w_{\tau_0}^\infty$, such that
$$ w_{\tau_0}^k \rightarrow w_{\tau_0}^\infty, \;\; \mbox{ as } k\rightarrow\infty \mbox{ uniformly in } \mathbb{R}^n.$$
It follows that the right hand side of inequality \eqref{shouliandao}
converges to
\begin{equation*}
 c'\int_{\mathbb{R}^n\setminus{B_{2}(0)}}\frac{G(-w_{\tau_0}^\infty(y))}{|y|^{n+sp}}dy.
\end{equation*}
 Combining \eqref{wwwjixian} with the continuity of $f$ and the fact $\varepsilon_k\rightarrow0$ ($k\rightarrow\infty$),
we see that the left hand side of inequality \eqref{shouliandao}
converges to $C\delta^{p(1-s)}$ ($0<\delta<1$) as $k\rightarrow\infty$. Because of
the arbitrariness of $\delta$, we derive that
\begin{equation*}
\int_{\mathbb{R}^n\setminus{B_{2}(0)}}\frac{G(-w_{\tau_0}^\infty(y))}{|y|^{n+sp}}dy\equiv0,
\end{equation*}
which implies that
\begin{equation}\label{contradictionqq}
u_{\tau_0}^\infty(x)\equiv u^\infty(x),\ \ \forall x\in\mathbb{R}^n\setminus{B_{2}(0)}.
\end{equation}
Recall that $u>0$ in $\Omega$ while $u(x)\equiv0$, $x\in\mathbb{R}^n\setminus\Omega$.
Since $x^k\in\Omega_M$, there exists $x^o$ such that $u^\infty(x^o)=0$,
then by \eqref{contradictionqq}, we have
\begin{equation}
0=u_{\tau_0}^\infty(x^o)=u^\infty(x^o + \tau_0 e_n)=u_{\tau_0}^\infty(x^o + \tau_0 e_n)= u^\infty(x^o + 2 \tau_0 e_n) = \cdots = u^\infty(x^o + m \tau_0 e_n).
\label{B10}
\end{equation}
By Theorem \ref{mainthm1133}, we deduce that
$$u^\infty(x^o + m \tau_0 e_n) \rightarrow 1  \mbox{ as } m\rightarrow\infty.$$

This is a contradiction to \eqref{B10}!
\medskip

\emph{Case 2.} Suppose that
$$\sup_{x\in \Omega_M}w_{\tau_0}(x)<0.$$
From Lemma \ref{holderestimate}, we know that $u$ is uniformly H\"{o}lder continuous in $\mathbb{R}^n$,
then for any $\eta\in(0,\tau_0)$ small enough, we get
\begin{equation}\label{mmmm}
\sup_{x\in \Omega_M}w_{\tau}(x)<0,\ \ \forall \tau_0-\eta<\tau\leq\tau_0.
\end{equation}
If there exists a constant $A_1>0$ such that
\begin{equation}\label{mmnn}
\sup_{x\in \mathbb{R}^n}w_\tau(x)= A_1,\ \ \forall \tau_0-\eta<\tau\leq\tau_0,
\end{equation}
then there exists a sequence $\{x^k\}$ in $\mathbb{R}^n$ such that
$$w_\tau(x^k)\rightarrow A_1,\ \text{as}\ k\rightarrow\infty.$$
Since $u=0$ in $\mathbb{R}^n\setminus{\Omega}$, it yields that
$$w_\tau(x)\leq 0,\ \ \ \forall x\in\mathbb{R}^n\setminus{\Omega}.$$
Thanks to Theorem \ref{mainthm1133} and \eqref{mmmm}, there exists an $M_1>M>0$ such that the sequence $\{x^k\}$ is contained in $\Omega_{M_1}\setminus
\overline{{\Omega_{M}}}$. For any $x\in\Omega_{M_1}\setminus
\overline{{\Omega_{M}}}$, we have
$$u(x),u_\tau(x)\in(t_1,1),$$
in which $f(\cdot)$ is non-increasing due to condition (c).

Then, similar to the argument as in \emph{Step 1}. We can derive
\begin{equation*}
\varepsilon_k C_\delta\geq \frac{c_{A_1}}{2}>0.
\end{equation*}
This is a contradiction as $\varepsilon_k$ $(k\rightarrow\infty)$ goes to zero. So, we have
$$w_\tau(x)\leq 0, \ \ \forall \tau_0-\eta<\tau\leq\tau_0, \ \forall x\in\mathbb{R}^n.$$
This contradicts the definition of $\tau_0$.

It follows that
$$w_\tau(x)\leq 0, \ \ \forall \tau\geq0, \ \forall x\in\mathbb{R}^n.$$
Moreover, by \emph{strong maximum principle} (Lemma \ref{Strongmax}), we arrive at \eqref{udandiao}.

This implies $u$ is strictly monotone increasing in $x_n$.
\medskip

\emph{Step 3.}
Now we prove the uniqueness.
Assume that $u$ and $v$
are two bounded solutions of \eqref{fractionalhold}. For $\tau\geq0$, denote
$$u_\tau(x):=u(x+\tau e_n),\ \ \tilde{w}_\tau(x):=v(x)-u_\tau(x),$$
where $e_n=(0,0,\cdots,1)$.

We first show that for $\tau$ sufficiently large,
\begin{equation}\label{taoweitt}
\tilde{w}_\tau(x)\leq0,\ \  \forall x\in\mathbb{R}^n.
\end{equation}
The proof of \eqref{taoweitt} is completely similar to \emph{Step 1} of the proof of monotonicity, so we omit the
details.
\eqref{taoweitt} provides a starting point from which we can decrease $\tau$ continuously as long
as \eqref{taoweitt} holds.

We prove that
\begin{equation}\label{udandiaoweiyi11}
\tilde{w}_\tau(x)\leq0,\ \forall \tau\geq0, \ \forall x\in\mathbb{R}^n.
\end{equation}
Define
$$\tau_0:=\inf\{\tau>0\mid \ \tilde{w}_\tau(x)\leq0,\  \forall x\in\mathbb{R}^n\}.$$
We show that
\begin{equation}\label{taoweiling11}
\tau_0=0.
\end{equation}
Suppose on the contrary $\tau_0>0$.
Similar to the argument of monotonicity in \emph{Step 2}, one can deduce that
\begin{equation}\label{mocontradictionqqweiyi}
v^\infty(x)\equiv u_{\tau_0}^\infty(x),\ \ \forall x\in\mathbb{R}^n\setminus{B_{2}(0)}.
\end{equation}
To finish the proof of the uniqueness, we need the following lemma.
\begin{lem}(Li and Zhang \cite{LiZ})\label{lemlizhang}
Let $\psi(x)=(1-|x|^2)_+^s$, then there exists constant $C$, such that
\begin{equation}
|(-\Delta)_p^s \psi(x)| \leq C, \;\; \forall \, x \in B_1(0).
\label{crrr}
\end{equation}
\end{lem}

Suppose that $z$ be a point on $\partial\Omega$. Without loss of generality, we may assume that there is
a ball $B \subset \Omega$ of radius 1 tangent to $\partial\Omega$ at point $z$. For simplicity of notation, we assume
that the center of the ball is the origin.

Let
$$D = \{x\in\Omega \mid dist(x,\partial\Omega)\geq2R_0\},$$
where $R_0$ is the same as in Lemma \ref{leme}.

We construct the sub-solution
$$\underline{u}(x)=u_D(x)+\varepsilon\psi(x),\ \ x\in B.$$
where $u_D:=u\cdot\chi_D$ and $\chi_D$ is defined as
\begin{equation*}
\chi_D(x)=\begin{cases}
&1, \ x\in D,\\
&0, \  x\in\mathbb{R}^n\setminus D.
\end{cases}\end{equation*}
It follows from \eqref{crrr}, Lemma \ref{leme} and Lemma \ref{appendix} in Appendix that for $x\in B$,
\begin{equation*}\begin{split}
(-\Delta)_p^s \underline{u}(x)&=(-\Delta)_p^s(u_D+\varepsilon\psi)(x)\\
&=C_{n,s,p} P.V.\int_{\mathbb{R}^n}\frac{G(u_D(x)+\varepsilon\psi(x)-u_D(y)-\varepsilon\psi(y))}{|x-y|^{n+sp}}dy\\
&=C_{n,s,p} P.V. \Big\{\int_{\mathbb{R}^n\setminus{B}}\frac{G(\varepsilon\psi(x)-u_D(y))-G(\varepsilon\psi(x))}{|x-y|^{n+sp}}dy\\
&\ \ +\int_{\mathbb{R}^n\setminus{B}}\frac{G(\varepsilon\psi(x))}{|x-y|^{n+sp}}dy
+\int_{B}\frac{G(\varepsilon\psi(x)-\varepsilon\psi(y))}{|x-y|^{n+sp}}dy\Big\}\\
&=\varepsilon^{p-1}(-\Delta)_p^s\psi(x)
+C_{n,s,p}\int_{D}\frac{G(\varepsilon\psi(x)-u(y))-G(\varepsilon\psi(x))}{|x-y|^{n+sp}}dy\\
&\leq \varepsilon^{p-1}C+2^{2-p}C_{n,s,p}\int_{D}\frac{G(-u(y))}{|x-y|^{n+sp}}dy\\
&\leq \varepsilon^{p-1}C-2^{2-p}\varepsilon_1^{p-1}C_{n,s,p}\int_{D}\frac{1}{|x-y|^{n+sp}}dy\\
&\leq \varepsilon^{p-1}C-2^{2-p}\varepsilon_1^{p-1}C_{n,s,p}C_{R_0}.\\
\end{split}\end{equation*}
One can choose $\varepsilon\leq\varepsilon_1\big[2^{2-p}C_{n,s,p}C_{R_0}C^{-1}\big]^{\frac{1}{p-1}}:=\varepsilon_0$ such that
$(-\Delta)_p^s \underline{u}(x)\leq 0$, $x\in B$. Then we fixed $\varepsilon=\frac{\varepsilon_0}{2}$, by Lemma \ref{lemll}
in Appendix, we arrive at
$$u(x)\geq\underline{u}(x)\geq \frac{\varepsilon_0}{2}\psi(x)=\frac{\varepsilon_0}{2}(1+|x|)^s(1-|x|)^s
\geq\frac{\varepsilon_0}{2} \delta^s(x),\ \ \forall x\in B.$$
For each fixed small $\delta_1\in \big(0,\min\{\tau_0,1\}\big)$, choosing $|x_n-z_n|=\delta_1$,
we derive that
\begin{equation}\label{udayu0}\begin{split}
u(x)&\geq\frac{\varepsilon_0}{2} \delta^s(x)\\
&\geq\frac{\varepsilon_0}{2} |x_n-z_n|^s\\
&=\frac{\varepsilon_0}{2} \delta_1^s>0.\\
\end{split}\end{equation}

Since $u$ is strictly monotone increasing in $x_n$, by \eqref{udayu0}, we infer that
\begin{equation}\label{maoduntt}
u_{\tau_0}(z)>\frac{\varepsilon_0}{2} \delta_1^s>0,\ \ \forall z\in\partial\Omega.
\end{equation}
Obviously, this property is preserved under translation. Let
$$\Omega^k = \{x \mid x+x^k \in \Omega\} \mbox{ and } \Omega^\infty = \lim_{k \rightarrow \infty} \Omega^k.$$

Taking a point  $x^0\in \partial\Omega^\infty$,  we deduce from \eqref{maoduntt} that
$$u_{\tau_0}^\infty(x^0)>0\ \ \text{but}\ \  v^\infty(x^0)=0.$$
This contradicts \eqref{mocontradictionqqweiyi}. So we must have $\tau_0=0$.
This proves \eqref{udandiaoweiyi11}, which implies that $v(x)\leq u(x)$.
Interchanging $u$ and $v$, we can also derive $u(x) \leq v(x)$. Therefore, we must have $u \equiv v$.
This yields the uniqueness.

This completes the proof of Theorem \ref{mainthmaaa}.
\end{proof}

\section{The proof of Theorem \ref{bankongjian}}
In this section, we consider a special case where $\Omega$ is an upper half space:
\begin{equation}\label{shangbankongjian11}\begin{cases}
(-\Delta)_p^s u(x)=f(u(x)),&x\in\mathbb{R}^n_+,\\
u(x)>0,&x\in\mathbb{R}^n_+,\\
u(x)=0, &x\in \mathbb{R}^n\setminus {\mathbb{R}^n_+}.
\end{cases}\end{equation}
\medskip

We are able to use the \emph{sliding method} in any direction to obtain a stronger result.

\begin{thm}\label{bankongjian11}
Suppose that $u\in C_{loc}^{1,1}\cap \mathcal{L}_{sp}$ be a bounded solution of \eqref{shangbankongjian11}.
Assume that  $f$ is continuous and satisfies condition (a)-(c) for some $0<t_0<t_1<\mu$.

Then $u$ is strictly monotone increasing in $x_n$, and moreover it depends on $x_n$ only.

Furthermore the bounded solution of \eqref{shangbankongjian11} is unique.
\end{thm}

\begin{proof}
For $\tau\geq0$, denote
$$u_\tau(x):=u(x+\tau \nu),\ \ w_\tau(x):=u(x)-u_\tau(x),$$
where $\nu=(\nu_1,\nu_2,\cdots,\nu_n)$ with $\nu_n>0$.

Similar to the proof of Theorem \ref{mainthmaaa},
for $\forall \tau>0$, we obtain
\begin{equation*}
w_\tau(x)<0,\ \ \forall x\in\mathbb{R}^n,
\end{equation*}
which implies that $u$ is strictly monotone increasing in any
direction $\nu=(\nu_1,\cdots,\nu_n)$ with $\nu_n>0$.

For each fixed point $x \in \mathbb{R}^n_+$, let $\nu_n \rightarrow 0$.
By the continuity of $\nabla u$, we deduce
that $\partial_\nu u(x) \geq 0$ for any $\nu$ with $\nu_n=0$. Replacing $\nu$ by  $-\nu$, we obtain
$\partial_\nu u=0$. Since this is true for all $\nu$ with $\nu_n=0$, we conclude that $u$ depends on $x_n$ only.

This completes the proof of Theorem \ref{bankongjian11}.
\end{proof}

\medskip
\section{Appendix}
In this section, we prove three lemmas.

\begin{lem}\label{appendix}
For $G(t)=|t|^{p-2}t$ ($p\geq2$), assume that $t_1+t_2>0$. Then
\begin{equation}\label{ggg}
G(t_1+t_2)\leq 2^{p-2}\big(G(t_1)+G(t_2)\big).
\end{equation}
\end{lem}
\begin{proof}
Without loss of generality, we may assume that $t_2>0$, let $t=\frac{t_1}{t_2}$.
Inequality \eqref{ggg} is equivalent to
$$G(1+t)\leq 2^{p-2}\big(G(1)+G(t)\big),\ t>-1.$$
Let
$$F(t)=G(1+t)- 2^{p-2}\big(G(1)+G(t)\big).$$
Assume that $F'(t)=0$, we can derive that $t=-\frac{1}{3}$ (local minimum point), $t=1$ (local maximum point), and
$$F(-\frac{1}{3})\leq 0,\ \ F(1)=0.$$
We also can calculate
$$\lim_{t\rightarrow-1^+} F(t)=0.$$
It follows that $F(t)\leq0$ for any $t>-1$.

This completes the proof of the lemma.
\end{proof}
\medskip

\begin{lem}\label{appendix2}
Assume that $u$
is the bounded solution of \eqref{fractionalbianjiewai} and $\Phi\in C_0^\infty(\mathbb{R}^n)$.
For any $0<\delta<1$, then
\begin{equation}\label{wwww}
|(-\Delta)_p^s(u+\varepsilon_k \Phi)(x)-(-\Delta)_p^su(x)|<\varepsilon_k C_\delta+C\delta^{p(1-s)}.
\end{equation}
\end{lem}
\begin{proof}
For any $x\in\mathbb{R}^n$, we divide the integral into two parts.
\begin{equation*}\begin{split}
&(-\Delta)_p^s(u+\varepsilon_k \Phi)(x)-(-\Delta)_p^su(x)\\
&=C_{n,s,p} P.V.\int_{\mathbb{R}^n}\frac{G(u(x)+\varepsilon_k\Phi(x)-u(y)-\varepsilon_k\Phi(y))-
G(u(x)-u(y))}{|x-y|^{n+sp}}dy\\
&=C_{n,s,p} P.V.\int_{B_{\delta}(x)}\frac{G(u(x)+\varepsilon_k\Phi(x)-u(y)-\varepsilon_k\Phi(y))-
G(u(x)-u(y))}{|x-y|^{n+sp}}dy\\
&\ \ \ +C_{n,s,p} \int_{\mathbb{R}^n\setminus{B_{\delta}(x)}}\frac{G(u(x)+\varepsilon_k\Phi(x)-u(y)-\varepsilon_k\Phi(y))-
G(u(x)-u(y))}{|x-y|^{n+sp}}dy\\
&:=I_\delta(x)+J_\delta(x).
\end{split}\end{equation*}
Now we first estimate $I_\delta(x)$. Applying the \emph{mean value theorem} to the function $G(t)=|t|^{p-2}t$,
we derive that there is a constant $C>0$ such that for any two quantities A and B, it holds
\begin{equation}\label{budengshi}
\big||A+B|^{p-2}(A+B)-|A|^{p-2}A\big|\leq C(|A|+|B|)^{p-2}|B|.
\end{equation}
Define
$$v(x):=u(x)+\varepsilon_k\Phi(x).$$
Since $u\in C_{loc}^{1,1}$ and $\Phi\in C_{0}^{\infty}(\mathbb{R}^n)$, by Taylor expansion, we have
$$v(x)-v(y)=\nabla v(x)\cdot (x-y)+O(|y-x|^2).$$
Let
$$A=\nabla v(x)\cdot (x-y)\ \ \text{and}\ \ B=O(|y-x|^2).$$
Then it follows from \eqref{budengshi} that
\begin{equation}\label{zuihouguji}\begin{split}
&\big||v(x)-v(y)|^{p-2}(v(x)-v(y))-|\nabla v(x)\cdot (x-y)|^{p-2}\nabla v(x)\cdot (x-y)\big|\\
&\leq C(|\nabla v(x)\cdot (x-y)|+|x-y|^2)^{p-2}|x-y|^2\\
&\leq C(|\nabla v(x)|^{p-2}|x-y|^p.
\end{split}\end{equation}
Then anti-symmetry of $\nabla v(x)\cdot (x-y)$ for $y\in B_\delta(x)$ implies
\begin{equation}\label{zuihouguji11}
P.V.\int_{B_{\delta}(x)}\frac{|\nabla v(x)\cdot (x-y)|^{p-2}\nabla v(x)\cdot (x-y)}{|x-y|^{n+sp}}dy=0.
\end{equation}
The estimate of the integral of $u$
on $B_\delta(x)$ is similar to the estimate of $v$, hence by virtue of \eqref{zuihouguji} and \eqref{zuihouguji11},
\begin{equation*}\begin{split}
&|I_\delta(x)|=C_{n,s,p} \Big|P.V.\int_{B_{\delta}(x)}\frac{G(u(x)+\varepsilon_k\Phi(x)-u(y)-\varepsilon_k\Phi(y))-
G(u(x)-u(y))}{|x-y|^{n+sp}}dy\Big|\\
& \ \ \ \ \ \ \ \ \ \leq C\big(|\nabla v(x)|^{p-2}+|\nabla u(x)|^{p-2}\big)
\int_{B_{\delta}(x)}\frac{|x-y|^p}{|x-y|^{n+sp}}dy\\
& \ \ \ \ \ \ \ \ \ \leq C\delta^{p(1-s)}.
\end{split}\end{equation*}
For the estimate of $J_\delta$, we have
\begin{equation*}\begin{split}
&J_\delta(x)=C_{n,s,p} \int_{\mathbb{R}^n\setminus{B_{\delta}(x)}}\frac{G(u(x)+\varepsilon_k\Phi(x)-u(y)-\varepsilon_k\Phi(y))-
G(u(x)-u(y))}{|x-y|^{n+sp}}dy\\
& \ \ \ \ \ \ \ =\varepsilon_kC_{n,s,p} \int_{\mathbb{R}^n\setminus{B_{\delta}(x)}}\frac{(\Phi(x)-\Phi(y))Q(x,y)}{|x-y|^{n+sp}}dy,
\end{split}\end{equation*}
where
$$Q(x,y)=(p-1)\big(\Phi(x)-\Phi(y)\big)\int_0^1|u(x)-u(y)+
t\varepsilon_k\big(\Phi(x)-\Phi(y)\big)|^{p-2}dt,$$
and we have used the following identity
$$|b|^{p-2}b-|a|^{p-2}a=(p-1)(b-a)\int_0^1|a+t(b-a)|^{p-2}dt.$$
Since $u$
is the bounded solution of \eqref{fractionalbianjiewai} and $\Phi\in C_0^\infty(\mathbb{R}^n)$, then
$$|J_\delta(x)|\leq \frac{C}{\delta^{sp}}\varepsilon_k:=C_\delta\varepsilon_k.$$

This completes the proof of the lemma.
\end{proof}

\begin{lem}(A comparison principle)\label{lemll}
Let $\Gamma$ be a bounded domain in $\mathbb{R}^n$.
Assume that $u,v\in C^{1,1}_{loc}(\Gamma)\cap\mathcal{L}_{sp}$
be lower semi-continuous on $\bar{\Gamma}$, and satisfy
\begin{equation}\label{comparison1}
\begin{cases}
(-\Delta)_p^s u(x)\geq(-\Delta)_p^s v(x), & x\in \Gamma,\\
u(x)\geq v(x), & x\in\mathbb{R}^n\setminus \Gamma.
\end{cases}\end{equation}
Then
\begin{equation}\label{comparison2}
u(x)\geq v(x), \ \ x\in\Gamma.
\end{equation}
If $u(x)=v(x)$ at some point $x\in\Gamma$, then
$$u(x)=v(x) \ \ \text{almost everywhere in} \ \mathbb{R}^n.$$
\end{lem}
\begin{proof}
Let
$$w(x)=u(x)-v(x).$$
Suppose \eqref{comparison2} is violated, then since $w$ is lower semi-continuous on $\bar{\Gamma}$, there exists
$x^0$ in $\Gamma$ such that
$$w(x^0)=\min_{\Gamma}w(x)<0.$$
It follows from the second inequality in \eqref{comparison1} that
\begin{equation*}\begin{split}
&(-\Delta)_p^s u(x^0)-(-\Delta)_p^s v(x^0)\\
&=C_{n,s,p} P.V.\int_{\mathbb{R}^n}\frac{G(u(x^0)-u(y))-G(v(x^0)-v(y))}{|x^0-y|^{n+sp}}dy\\
&\leq C_{n,s,p} \int_{\mathbb{R}^n\setminus\Gamma}\frac{G(u(x^0)-u(y))-G(v(x^0)-v(y))}{|x^0-y|^{n+sp}}dy\\
&<0.\\
\end{split}\end{equation*}
This contradicts the first inequality in \eqref{comparison1} and hence \eqref{comparison2} must be true. It follows that
if $w(x^0)=0$ at some point $x^0\in\Gamma$, then
\begin{equation*}\begin{split}
&(-\Delta)_p^s u(x^0)-(-\Delta)_p^s v(x^0)\\
&=C_{n,s,p} P.V.\int_{\mathbb{R}^n}\frac{G(u(x^0)-u(y))-G(v(x^0)-v(y))}{|x^0-y|^{n+sp}}dy\\
&\leq0,\\
\end{split}\end{equation*}
while on the other hand, from the first inequality in \eqref{comparison1}, we should have
$$(-\Delta)_p^s u(x^0)-(-\Delta)_p^s v(x^0)\geq0$$
and hence the integral must be 0. Taking into account that $w$ is already nonnegative, we
derive
$$w(x)=0 \ \ \text{almost everywhere in} \ \mathbb{R}^n.$$

This proves the lemma.
\end{proof}

\medskip

\section*{Acknowledgements}
This work was completed while the first author was
visiting Yeshiva University. He would like to thank the Department of Mathematical Science
at Yeshiva University for the hospitality and the stimulating environment.

\end{document}